\documentclass[a4paper,11pt]{article}

\input{styleart.sty}

\usepackage[colorinlistoftodos]{todonotes}

\DeclareMathOperator*{\T}{T}
\DeclareMathOperator*{\NE}{ne}
\newcommand{\nephi}{\varphi_{\NE}}

\newcommand{\demph}[1]{\emph{#1}}

\title{Nonequational Stable Groups}
\author{Isabel M\"{u}ller \, 
and Rizos Sklinos}
\begin{document}
 \maketitle
\abstract{We introduce a combinatorial criterion for verifying whether a formula is not the conjunction of an equation and a co-equation. 
Using this, we give a proof for the nonequationality of the free group.

Furthermore, we generalize the latter result to the first-order theory of any free product of groups of the form $G*\F_{\omega}$.}

\section{Introduction}
The notion of equationality has been introduced by Srour \cite{MR2633841} and further developed by Pillay-Srour \cite{MR771800}. It is best understood intuitively 
as a notion of Noetherianity on instances of first-order formulas (see section \ref{sec:Criterion} for a formal definition). 
A first-order theory is equational when every first-order formula is equivalent to a boolean combination of equations.  
As it is often the case in model theory, equationality is modeled after a phenomenon in algebraically closed fields. There,  
every first-order formula is a boolean combination of varieties, i.e. closed sets in the Zariski topology, which in turn is Noetherian. 

The equationality of a first-order theory implies another fundamental property: any equational first-order theory is stable. 
Stability had been introduced by Shelah as the first dividing line in his classification program. A first-order theory is stable if it admits a nicely behaved independence relation. We note that, in general, 
having an abstract independence relation does not imply a notion of dimension that can be used to prove a descending chain condition 
as in the case of algebraically closed fields. Despite that, at the time equationality was introduced there was no known example of a stable nonequational theory. A few years later Hrushovski and Srour \cite{HrushovskiSrour} produced the first such "artificial" example by tweaking the free pseudospace, 
a structure introduced by Baudisch and Pillay which is $2$-ample but not $3$-ample. 

In 2006 Sela \cite{MR3034289} proved that torsion-free hyperbolic groups are stable. This 
is considered by many one of the deepest results in the model theory of groups. 
In addition, Sela proved, in a yet unpublished paper \cite{SelaEquational}, that torsion-free hyperbolic groups are 
nonequational, whence these theories are the first natural examples of stable nonequational theories. 
His proof of nonequationality relies on the heavy machinery introduced and used in a series of papers on the 
elementary theory of free groups culminating to the positive answer of Tarski's question on whether or not nonabelian free groups share the 
same first-order theory. 

The sophisticated methods introduced for tackling Tarski's problem also allowed Sela to answer a long standing question of Vaught . He proved, in a yet unpublished paper \cite{SelaFreeProducts},  
that whenever $G_1$ is elementarily equivalent to $G_2$ and $H_1$ is elementarily equivalent to $H_2$, then the free product $G_1*H_1$ 
is elementarily equivalent to $G_2*H_2$. In addition, he proved that whenever $G*H$ is a nontrivial free product, which is not $\Z_2*\Z_2$, 
then it is elementarily equivalent to $G*H*\F$ for any free group $\F$.

In this paper we give an elementary transparent proof for the nonequationality of the 
first-order theory of nonabelian free groups. We use (essentially) the same first-order formula as Sela, 
but our arguments avoid his complicated machinery. As a matter of fact, our proof generalizes to any free product of groups of the form $G*\F_{\omega}$. The main result of this paper is: 

\begin{thmIntro} 
The first-order theory of the free product of groups $G*\F_{\omega}$ is nonequational. 
\end{thmIntro}

Since $\F_{\omega}$ is a model of the theory of the free group (see \cite[Theorem 4]{MR2238945} and \cite{MR2293770}), we get as a corollary that the first-order theory of the free group is nonequational.

Taking into account the results in \cite{SelaFreeProducts}, our theorem implies: 

\begin{thmIntro}
Let $G_1*G_2$ be a nontrivial free product of groups, which is not $\Z_2*\Z_2$. Then its first-order theory is nonequational. 
\end{thmIntro}

Sela, in the work mentioned above \cite{SelaFreeProducts}, proved that a free product of stable groups remains stable.
This fact together with the main result of this paper give an abundance of new 
stable nonequational theories. 

%

\section{The Criterion}\label{sec:Criterion}

In the following section we will introduce the notion of equationality and present a criterion which implies 
that a formula is not the conjunction of an equation and a co-equation. We furthermore will argue that under some 
slightly stricter conditions, this amounts to show that a first-order theory is nonequational. 
\ \\ \\
{\bf Conventions:} In this paper a first-order theory is always assumed to be complete. In addition, all parameters are coming 
from a fixed saturated model $\mathbb{M}$, the "monster model", unless otherwise specified. The notation "$\models$" will be used for "$\mathbb{M}\models$".
\ \\ \\
\begin{definition}\label{def:equationalityformula}
Let $T$ be a first-order theory. Then a formula $\varphi(x, y)$ is an \demph{equation} in the tuple $x$ if any 
collection of instances of $\varphi(x, y)$ is equivalent (modulo $T$) to a finite subcollection.   
\end{definition}

An easy example of an equation in an arbitrary theory $T$ is an actual equation, i.e. a formula of the type $x=y$. 
Note that being an equation is not closed under boolean combinations, as the formula $x\neq y$ is not an equation in any 
theory $T$ with infinite models.

\begin{definition}\label{def:equationalitytheory}
A first-order theory $T$ is $n$\demph{-equational} if every formula $\varphi(x, y)$ with variable length $\abs{x}=n$ 
is a boolean combination of equations. 
Moreover, the first-order theory $T$ is \demph{equational} if it is $n$-equational for all $n<\omega$.
\end{definition}

It remains an interesting open question, whether being $1$-equational implies being equational. 
All known examples of stable nonequational theories, i.e. the examples in this paper and Hrushovski's tweaked free pseudospace, 
are in fact not $1$-equational.

The following fact gives a more combinatorial flavor to the definition of equationality for a formula.

\begin{fact}\label{fact:noequation}
Let $T$ be a first-order theory. Then a first-order formula $\varphi(x,y)$ is {\bf not} an equation (in $T$) if and only if  for arbitrarily large $n\in\N$, there are tuples $(a_i)_{i\leq n}, (b_j)_{j\leq n}$ such that 
$\models\varphi(a_i,b_j)$ for all $i<j$, but $\not\models\varphi(a_i,b_i)$.
\end{fact}

The following remark will have interesting consequences in our study of nonequational theories.

\begin{remark}\label{rem:dnf}
Assume $\varphi(x,y)$ to be a formula equivalent to a boolean combination of equations. Then $\varphi(x,y)$ is equivalent to a formula of the form
   \begin{equation*}
   \bigvee_{0\leq i\leq n}(\psi_1^i(x,y)\wedge \neg\psi_2^i(x,y)),
   \end{equation*}
for some equations $\psi_1^i, \psi_2^i$ and $n\in \N$. This follows easily from the facts that every formula is equivalent to a formula in disjunctive normal form and that finite disjunctions and finite conjunctions of equations are again equations (see \cite[Lemma 2.8]{MR771800}). 
\end{remark}

In the following we will give a combinatorial criterion for formulas to not be of the form 
$\psi_1(x,y)\wedge \neg\psi_2(x,y)$, where $\psi_1$ and $\psi_2$ are equations. 

\begin{lemma}\label{lem:criterion}
Let $\T$ be a first-order theory and $\varphi(x,y)$ be a formula. 
If for arbitrary large $n\in\N$ there exist matrices $A_n:=(a_{ij})_{i,j\leq n}$ and $B_n:=(b_{kl})_{k,l\leq n}$ such that 
\begin{equation*}\label{eq:necondition}
  \models \varphi(a_{ij},b_{kl}) \text{ if and only if }i\neq k\text{ or }(i,j)=(k,l),
 \end{equation*}
then $\varphi(x,y)$ is not equivalent to a formula of the form $\psi_1(x,y)\wedge \neg\psi_2(x,y)$, where $\psi_1$ and $\psi_2$ are equations. 
\end{lemma}
\begin{proof}
 We first show that under the hypothesis of the lemma, every row witnesses that the formula $\neg\varphi(x,y)$ is not equivalent to an equation. To see that, fix some $i_0$ and note that with $a_j:=a_{i_0j}$ and $b_l:=b_{i_0l}$, we get 
 \begin{equation*}
  \models\neg\varphi(a_j,b_l) \text{ for all }j<l\text{ and }\not\models\neg\varphi(a_j,b_j),
 \end{equation*}
whence by Fact \ref{fact:noequation} the formula $\neg\varphi(x,y)$ is not equivalent to an equation. 

Now, aiming for a contradiction, assume $\varphi(x,y)\equiv \psi_1(x,y)\wedge\neg\psi_2(x,y)$, for some equations 
$\psi_1,\psi_2$, whence in particular $\neg\varphi(x,y)\equiv \neg\psi_1(x,y)\vee\psi_2(x,y)$. 
If there was some index $i_0$, such that $\models\psi_1(a_{i_0j},b_{i_0l})$ for all $j,l$, 
then for that row we would have $\neg\varphi(a_{i_0j},b_{i_0l})\leftrightarrow \psi_2(a_{i_0j},b_{i_0l})$ for all $j,l\leq n$, 
contradicting the fact that $\psi_2$ is an equation. Thus, for any index $i\leq n$, there exists some $j_i,l_i$ such that 
$\models\neg\psi_1(a_{ij_i},b_{il_i})$. Set $a_i:=a_{ij_i}, b_k:=b_{kl_k}$. Note that for $i\neq k$ we have $\models\varphi(a_i,b_k)$, whence 

\begin{equation*}
 \models\psi_1(a_i,b_k)\text{ for all }i<k\text{ and }\not\models\psi_1(a_i,b_i),
\end{equation*}
contradicting the fact that $\psi_1$ is an equation.
\end{proof}

Our method provides a general criterion for proving that a first-order theory is nonequational, given in the following proposition. There, we will 
use the notion of a \demph{type} which is a maximally consistent set of formulas. We remark that two tuples in the same orbit under the automorphism group of the structure have the same type. For the applications of this paper, in Sections \ref{sec:freegroup} and \ref{sec:FreeProducts}, we will actually use elements in the same automorphism orbit, hence the reader unfamiliar with the notion of "having the same type" can, on first reading, substitute it with the stronger notion of "being in the same automorphism orbit".   

\begin{proposition}\label{prop:criterion}
Let $T$ be a first-order theory. Suppose there exist a formula $\varphi(x,y)$ and arbitrarily large matrices 
$A_n, B_n$ such that: 
\begin{itemize}
 \item[(i)] there exists a type of $T$ which is satisfied by any tuple $(a_{ij},b_{kl})$ of entries of the matrices for $i\neq k$ and for $(i,j)=(k,l)$;
 \item[(ii)] the formula $\varphi(x,y)$ is satisfied by $(a_{ij},b_{kl})$ if and only if $i\neq k$ or $(i,j)=(k,l)$.
\end{itemize}
Then $T$ is nonequational.
\end{proposition}
\begin{proof}
 We will show that $\varphi(x,y)$ is not in the boolean algebra generated by equations. Otherwise, by Remark \ref{rem:dnf}, 
 there existed $m\in \N$ and equations 
 $\psi_1^i(x,y),\psi_2^i(x,y)$ for $i\leq m$ such that 
 
\begin{equation*}
\varphi(x,y)\equiv \bigvee_{0\leq i\leq m}(\psi_1^i(x,y)\wedge \neg\psi_2^i(x,y)).
\end{equation*}

By the hypothesis, for any $n$, we have $\models\varphi(a_{11},b_{11})$, where $a_{11}=A_n(1,1)$ and $b_{11}=B_n(1,1)$ (to be formal $a_{11}=a_{11}(n)$ and $b_{11}=b_{11}(n)$, but for notational simplicity we drop the $n$). 
Thus, there exists some 
$i\leq m$ such that $\models\psi^i_1(a_{11},b_{11})\wedge\neg\psi^i_2(a_{11},b_{11})$ for infinitely many $n$. 

We set 
$\theta(x,y):=\psi^i_1(x,y)\wedge\neg\psi^i_2(x,y)$. 
As all $(a_{ij},b_{kl})$ for $(i,j)=(k,l)$ or $i\neq k$ have the same type as $(a_{11},b_{11})$, we get that 
$\models\theta(a_{ij},b_{kl})$ for all previously mentioned indices (and arbitrarily large $n$). 
On the other hand, if $i=k$, but $j\neq l$, then $\models \neg\varphi(a_{ij},b_{kl})$, 
whence in particular $\models\neg \theta(a_{ij},b_{kl})$ (and all $n$). By Lemma \ref{lem:criterion}, 
this contradicts the fact that $\theta(x,y)$ is the conjunction of an equation and a co-equation. 
\end{proof}

\section{Nonequationality of the Free Group}\label{sec:freegroup}

In the following we will show that the theory of the free group is not $1$-equational, 
and hence not equational. The next result (see \cite[Theorem 4]{MR2238945} and \cite{MR2293770}) allows us to work in $\F_{\omega}$, the free group of rank $\omega$. 


\begin{fact} 
The following chain of free groups under the natural embeddings is elementary.
$$\F_2\subset \F_3\subset\ldots\subset\F_n\subset\ldots$$
\end{fact}

Working in $\F_{\omega}$, for which we fix a basis $\{e_1, e_2, \ldots, e_n, \ldots\}$, we will show that 
the following formula is not equivalent to a boolean combination of equations:

\begin{equation*}\label{eq:neformula}
\nephi(x,y):= \forall u,v ([u,v]\neq 1\rightarrow xy\neq u^5v^4).
\end{equation*}

Our formula is a variation of the formula Sela uses: 
\begin{equation*}\label{eq:Seformula}
\phi_S(x,y):= \exists u,v ([u,v]\neq 1\land yx= u^{10}v^{-9}). 
\end{equation*}

We recall some useful group theoretic facts. An element of a free group is called {\em primitive}  
if it is part of some basis of the free group. 

\begin{fact}\label{primitive}
Let $a$ be a primitive element of $\mathbb{F}$. Suppose $a$ belongs to a subgroup  $H$ of $\mathbb{F}$. 
Then $a$ is a primitive element of $H$.    
\end{fact}

The above fact can be obtained as a corollary of the Kurosh Subgroup Theorem \cite[Chapter IV, Theorem 1.10]{MR1812024}). We also recall an easy corollary of \cite[Chapter I, Proposition 4.17]{MR1812024}):

\begin{fact}\label{nonprimitive}
Let $e_1,\ldots, e_n$ be a basis of the free group $\F_n$ of rank $n$. Then, for any $|m_i|\neq 1$, $e_1^{m_1}\cdot e_2^{m_2}\ldots e_n^{m_n}$ is not a primitive element. 
\end{fact}

With Fact \ref{primitive} and Fact \ref{nonprimitive} at hand, we now can prove:

\begin{lemma}\label{lem:phigeneric}
For any pair $(a,b)$ which is part of some 
basis of $\F_{\omega}$ we have $\F_{\omega}\models\nephi(a,b)$.
\end{lemma}
\begin{proof}
It is enough to prove that $e_1e_2$ is not a product of a fifth and a fourth power of any two elements of $\F_{\omega}$ that do not commute. 
Suppose otherwise that there are $u,v$ two elements of $\F_{\omega}$ such that $u^5v^4=e_1e_2$ and $[u,v]\neq 1$. Recall 
that any two elements that do not commute generate a free subgroup of rank $2$. Fact \ref{primitive} yields that $e_1e_2$ is a primitive element of $\langle u, v\rangle$. On the other hand since $u^5v^4=e_1e_2$, we get, by Fact \ref{nonprimitive}, that the element $e_1e_2$ is not primitive in $\langle u, v\rangle$, 
a contradiction. 
\end{proof}

We will use Proposition \ref{prop:criterion} in order to prove that the theory of the free group is nonequational.  
Therefore, for arbitrary $n\in \N$, consider the following matrices:
 \begin{displaymath}
\mathbf{A_{n}} =
\left( \begin{array}{cccc}
e_2^5e_1 & e_3^5e_1 & \ldots & e_{n+1}^5e_1 \\
e_3^5e_2 & e_4^5e_2 & \ldots & e_{n+2}^5e_2 \\
\vdots & \vdots & \ddots \\
e_{n+1}^5e_n & e_{n+2}^5e_n & \ldots & e_{2n}^5e_n
\end{array} \right), \ \ 
\mathbf{B_{n}} =
\left( \begin{array}{cccc}
e_1^{-1}e_2^{-4} & e_1^{-1}e_3^{-4} & \ldots & e_1^{-1}e_{n+1}^{-4} \\
e_2^{-1}e_3^{-4} & e_2^{-1}e_4^{-4} & \ldots & e_2^{-1}e_{n+2}^{-4} \\
\vdots & \vdots & \ddots \\
e_n^{-1}e_{n+1}^{-4} & e_n^{-1}e_{n+2}^{-4} & \ldots & e_n^{-1}e_{2n}^{-4}
\end{array} \right)
\end{displaymath}

We will see that $\nephi(x,y)$ together with the matrices $A_n,B_n$ satisfy the hypotheses of Proposition \ref{prop:criterion}.

\begin{lemma}\label{lem:entriesindependent}
Let $A_n=(a_{ij})$ and $B_n=(b_{kl})$ be the matrices given above. 
If $i\neq k$ or $(i,j)=(k,l)$, then $a_{ij}$ and $b_{kl}$ form part of a basis of $\F_{\omega}$. 
\end{lemma}

\begin{proof}   
Consider first $a_{ij}\in A_n$ and $b_{kl}\in B_n$ arbitrary with $i\neq k$. Then
\begin{equation*}
 a_{ij}=e_{i+j}^5e_i \text{ and } b_{kl}=e_{k}^{-1}e_{k+l}^{-4}.
\end{equation*}
Consider the sets $A:=\{i,k,i+j, k+l\}$ and $S:=A\setminus\{i,k\}$. 
Then the set $\{e_s\mid s\in S\}\cup\{a_{ij},b_{kl}\}$ is part of a basis, 
as the subgroup it generates contains the following part of a basis $\{e_i,e_k\}\cup\{e_s\mid s\in S\}$ which has the same size. 

If $(i,j)=(k,l)$, then the set $\{a_{ij},b_{ij}\}=\{e_{i+j}^5e_i,e_{i}^{-1}e_{i+j}^{-4}\}$ forms a basis of $\F_2$, 
as the subgroup it generates contains the following part of a basis $\{e_i,e_{i+j}\}$ which has the same size. 
\end{proof}

\begin{lemma}\label{lem:not2equational}
Let $A_n=(a_{ij})$ and $B_n=(b_{kl})$ be the matrices given above. Then $\lnot\nephi$ is satisfied in 
$\F_{\omega}$ by any pair $(a_{ij},b_{kl})$ if $i=k$ and $j\neq l$. 
\end{lemma}
\begin{proof}
Consider $a_{ij}\in A_n$ and $b_{il}\in B_n$ arbitrary for $j\neq l$. Then 
\begin{equation*}
  a_{ij}b_{kl}=e_{i+j}^5e_ie_{i}^{-1}e_{i+l}^{-4}=e_{i+j}^5e_{i+l}^{-4}.
\end{equation*}
Cleary $u=e_{i+j}$ and $v=e^{-1}_{i+l}$ do not commute for $j\neq l$, whence 
$\F_{\omega}\models\lnot\nephi(a_{ij},b_{il})$, as desired.
\end{proof}

We can now prove the nonequationality of the free group.

\begin{theorem}
The theory of the free group is nonequational. 
\end{theorem}
\begin{proof} 
We confirm that the hypotheses of Proposition \ref{prop:criterion} are satisfied for the first-order formula $\nephi(x,y)$ 
and the matrices $A_n$, $B_n$ given above. Indeed, by Lemma \ref{lem:entriesindependent} the pairs 
$(a_{ij},b_{kl})$ for $i\neq k$ and for $(i,j)=(k,l)$ all satisfy the same type, namely the type $tp(e_1,e_2)$. Thus, the 
first condition of Proposition \ref{prop:criterion} holds. 

For the second condition, we need to prove that $\F_{\omega}\models\nephi(a_{ij},b_{kl})$ if and only if 
$i\neq k$ or $(i,j)=(k,l)$. The right to left direction follows from lemmata \ref{lem:phigeneric} and \ref{lem:entriesindependent}, while the other 
direction is Lemma \ref{lem:not2equational}.
\end{proof}

\begin{remark}
Any formula $\psi_{m}(x,y):= \forall u,v ([u,v]\neq 1\rightarrow xy\neq u^mv^{m-1})$, for $m>2$, works with  essentially the same proof and witnesses that the free group is nonequational. 

We arbitrarily chosen $m=5$, as it seems to be the smallest number that makes the proof of Lemma \ref{lem:phigeneric2}, in the next section, (easily) work. As matter of fact, Lemma \ref{lem:phigeneric2} also works for any $\psi_m(x,y)$ with $m>4$ with similar proof.  
\end{remark}

\section{Free Products of Groups}\label{sec:FreeProducts}
In this section we generalize the main result of Section \ref{sec:freegroup} 
to the first-order theory of any free product of groups of the form  $G*\F_\omega$. 

\begin{theorem}\label{FreeProducts}
The first-order theory of the free product of groups $G*\F_\omega$ is nonequational.
\end{theorem}

In order to prove Theorem \ref{FreeProducts}, we only need to show  that Lemma \ref{lem:phigeneric} is still valid for the group $G*\F_\omega$. We will freely use some Bass-Serre theory (\cite{MR1954121}) and properties of isometric actions (\cite[Chapter 3]{MR1851337}). For the unfamiliar reader we collect next some useful results:

\begin{fact}\label{BassSerre}
Let $A, B$ be nontrivial groups. The group $A*B$ acts by simplicial automorphisms without inversions on a simplicial tree, say $T$, and the action has the following properties:
\begin{itemize}
\item Vertices of the tree are in correspondence with $(A*B)/A\sqcup (A*B)/B$. Moreover, each vertex is stabilized either by a conjugate of $A$ or by a conjugate of $B$ (\cite[Chapter I, Theorem 7]{MR1954121}). 
\item Every edge is trivially stabilized (\cite[Chapter I, Theorem 7]{MR1954121}). 
\item Every nontrivial element $g$ is either elliptic, i.e. it fixes a unique point in the tree or it is hyperbolic, i.e. there exists a unique line $Ax(g)$, called the axis of the element, on which it acts by translation (see \cite[Chapter 3, Theorem 1.4]{MR1851337}). By a line we mean a bi-infinite path (without backtracking) (see \cite[Chapter I, Definition 2]{MR1954121}).
\item If $g$ is hyperbolic, then $g^m$, for any integer $m\neq 0$, is hyperbolic too, moreover $Ax(g)=Ax(g^m)$ and $tr(g^m)=|m|\cdot tr(g)$ (see \cite[Chapter 3, Lemma 1.7(3)]{MR1851337}). By $tr(g)$ we denote the translation length which is defined by $min\{d_T(x,g.x) \ : \ x\in V(T)\}$, where $V(T)$ is the set of vertices of the tree $T$ and $d_T(x,y)$ is the (integer) distance between the vertices $x,y$ given by the length of the (unique) shortest path between them.    
\item If $g, h$ are both elliptic and they fix distinct vertices, then $gh$ is hyperbolic (\cite[Chapter 3, Lemma 2.2]{MR1851337}).
\item If $g,h$ are both hyperbolic and either:
\begin{itemize}
    \item[(i)] their axes do not meet; or
    \item[(ii)] their axes meet coherently, i.e. they intersect nontrivially and either they meet at a vertex or they translate in the same direction along their intersection.
   \end{itemize}
 Then $gh$ is hyperbolic as well (see \cite[Chapter 3, Lemma 2.2]{MR1851337} for (i), \cite[Chapter 3, Lemma 3.1]{MR1851337} for (ii)).
\item If $g,h$ are both hyperbolic, their axes meet incoherently and $|Ax(g)\cap Ax(h)|< min \{tr(g), tr(h)\}$, then $tr(gh)=tr(g)+tr(h)-2|Ax(g)\cap Ax(h)|$ (see \cite[Chapter 3, Lemma 3.3(2)]{MR1851337}). By $|Ax(g)\cap Ax(h)|$ we denote the length of the path $Ax(g)\cap Ax(h)$ when the path is finite or $\infty$ (which is larger than any natural number) when it is not.  
\item If $g,h$ are both hyperbolic, their axes meet incoherently and $|Ax(g)\cap Ax(h)|> min \{tr(g), tr(h)\}$, then, assuming that $tr(h)\leq tr(g)$,  $tr(gh)=tr(g)-tr(h)$ (see \cite[Chapter 3, Lemma 3.4(2)]{MR1851337}).
\item If $g,h$ are both hyperbolic, their axes meet incoherently, $|Ax(g)\cap Ax(h)|= min \{tr(g),$ $tr(h)\}$, and $tr(gh)=0$, then $A(gh)\cap Ax(g)$ is a single vertex, where $A(gh)$ denotes the fix point set of $gh$ (see \cite[Chapter 3, Lemma 3.5(3)]{MR1851337}). In particular, the vertex fixed by $gh$ lies in the axis of $g$.
\end{itemize}
\end{fact}

For convenience of notation, if $x$ is a vertex in $T$ and $S$ is a subtree we will denote by $pr_{S}(x)$ the unique vertex in $S$ that minimizes the distance between $x$ and $S$. In addition, if $x,y$ are vertices in $T$ we will denote by $[x,y]$ the (unique) shortest path connecting them. We will sometimes refer to it as the geodesic segment between $x$ and $y$. 

It is also worth noting that the action respects distances. We are now ready to prove:

\begin{lemma}\label{lem:phigeneric2}
For any pair $(a,b)$ which is part of some basis of $\F_{\omega}$ we have that $G*\F_\omega\models\nephi(a,b)$.
\end{lemma}
\begin{proof}
It will be enough to prove that $e_1e_2$ is not a product of a fifth and a fourth power of elements in $G*\F_{\omega}$ that do not commute. We may assume that $G$ is nontrivial. We consider the Bass-Serre tree that corresponds to the free splitting $G*\F_{\omega}$.

Suppose, for a contradiction, that we can find $u$ and $v$ in $G*\F_{\omega}$ such that $[u,v]\neq 1$ and $u^5v^4=e_1e_2$. We take cases for $u$ (which is necessarily nontrivial):
\begin{itemize}
    \item Suppose $u$ is elliptic. Then $v$ must be elliptic as well. Indeed, suppose, for a contradiction, that $v$ is hyperbolic and $*$ is the vertex stabilized by $\F_{\omega}$. Then $u^5v^4$ must fix this vertex, but since $v^4$ acts hyperbolically, the geodesic connecting $u^5v^4.*$ and $v^4.*$ consists of the concatenation of the following geodesic segments (see \cite[Chapter 3, Theorem 1.4]{MR1851337}) 
    $$[*, pr_{Ax(v)}(*)]\sqcup[pr_{Ax(v)}(*),v^4.pr_{Ax(v)}(*)]\sqcup [v^4.pr_{Ax(v)}(*), v^4.*] $$
    Since $u^5v^4.*=*$ the (unique) vertex fixed by $u$, say $x$, must lie in the middle of the above geodesic segment and since the lengths of the two side paths, i.e. $[*, pr_{Ax(v)}(*)]$ and $[v^4.pr_{Ax(v)}(*), v^4.*]$ are equal, it must lie in the middle of $[pr_{Ax(v)}(*),$ $v^4.pr_{Ax(v)}(*)]$. Moreover, since $u^5$ acts as a reflection through $x$ sending $[v^2.pr_{Ax(v)}(*), v^4.pr_{Ax(v)}(v)]$ to $[v^2.pr_{Ax(v)}(*), *]$, we have $u^{5}.v^4.pr_{Ax(v)}(*)=pr_{Ax(v)}(*)$. In particular, $u^5.v^4$ fixes $[*,pr_{Ax(v)}(*)]$ (see Figure \ref{Elliptic-elliptic}). Thus, $*=pr_{Ax(v)}(*)$ and consequently 
    $*$ belongs to the axis of $v$. Hence, $u^5$ belongs to $v^2\F_\omega v^{-2}$ (and also $u$, since $u^5$ is nontrivial) and therefore $v^2\gamma^5 v^{-2}v^4=v^2\gamma^5 v^2=e_1e_2$ for some nontrivial $\gamma\in\F_\omega$. Now,  we claim that $e_1e_2.v.*\neq \gamma^5.v.*$. Indeed, otherwise $(e_1e_2)^{-1}\gamma^5$ would fix $v.*$, but it also fixes $*\neq v.*$ and since it is not trivial it cannot fix a segment. We notice that since $\gamma^5$ is equal to $v^{-2}u^5v^{2}$, it moves $v^2.*$ to $v^{-2}.*$ while fixing $*$. Hence, $\gamma^5$ must move $v.*$ to $v^{-1}.*$. Thus, $e_1e_2$ moves $v.*$ out of the axis of $v$ and therefore $e_1e_2.v^{-2}$ moves $v.*$ out of the axis as well. Indeed, if not, then $e_1e_2$ moves $v^{-1}.*$ to $v.*$, but then $\gamma^5.e_1e_2$ fixes $v^{-1}.*$, a contradiction. On the other hand,  $v^2\gamma^5$ fixes $v.*$, hence $v^2\gamma^5 v^2\neq e_1e_2$.     
    \begin{figure}[ht!]
    \centering
    \includegraphics[width=.6\textwidth]{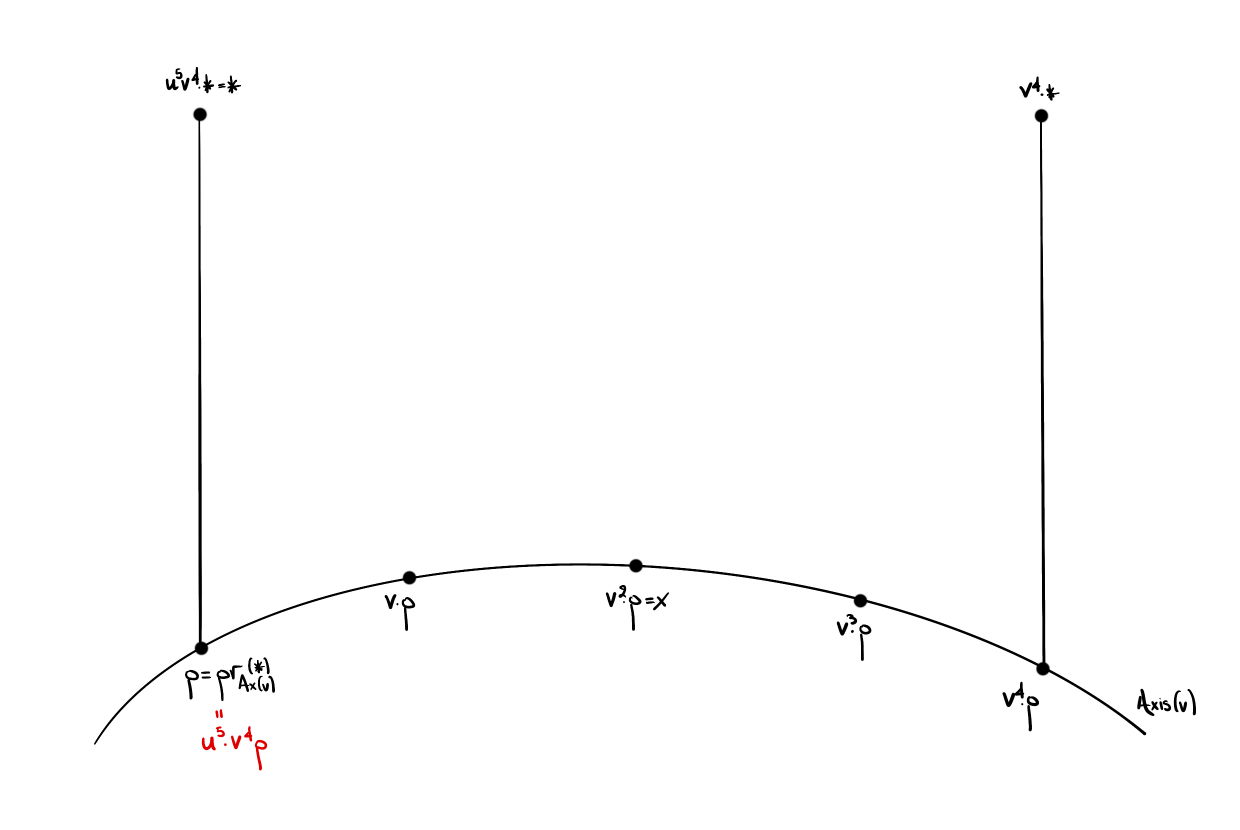}
    \caption{The elliptic case.}
    \label{Elliptic-elliptic}
    \end{figure}
    
    Since, in this case, both $u$ and $v$ are elliptic, they must fix the same vertex, otherwise $u^5v^4$ would be hyperbolic. And, finally, since $u^5v^4$ belongs to $\F_{\omega}$, the common fixed vertex must be the vertex stabilized by $\F_{\omega}$. But this contradicts Lemma \ref{lem:phigeneric}. 
    \item Suppose $u$ is hyperbolic. Then $v$ must be hyperbolic as well. Indeed, suppose for a contradiction that $v$ is elliptic, then a similar argument to the argument of the previous case gives that $*$ (the vertex stabilized by $\F_{\omega}$) lies in the axis of $u$ and $v^4$ acts as a reflection through the midpoint of $[*, u^{-5}.*]$. In particular,  $v^4u^3$ fixes $u^{-4}.*$. Hence, $v^4u^3=u^{-4}\gamma u^4$, for some nontrivial $\gamma\in\F_{\omega}$. Since $\gamma$ is equal to $u^4v^4u^{-1}$, it moves $u^{-1}.*$ to $u.*$ while fixing $*$. This is because $v^4$ reflects $u^{-2}.*$ to $u^{-3}.*$. But, on the other hand, $\gamma$ also moves $u.*$ to $u^{-1}.*$. Therefore, $\gamma^2$ fixes both $*$ and $u.*$, thus it is trivial and consequently $\gamma$ is trivial, a contradiction.
    
    Therefore, both $u$ and $v$ are hyperbolic. In this case, the two axes, $Ax(v)$ and $Ax(u)$, must meet incoherently. We take cases according to how the length of their intersection compares with the minimum of translation lengths. 
    \begin{itemize}
        \item[(i)] Suppose $|Ax(u)\cap Ax(v)|>min\{tr(u^5), tr(v^4)\}$. If $tr(v^4)<tr(u^5)$, then $tr(u^5v^4)=tr(u^5)-tr(v^4)>0$ and consequently $u^5v^4$ is hyperbolic. 
        
        If $tr(u^5)<tr(v^4)$, then, symmetrically, $tr(v^{-4}u^{-5})=tr(v^4)-tr(u^5)>0$, hence  $v^{-4}u^{-5}$ is hyperbolic, which implies that $u^{5}v^4$ is hyperbolic as well. 
        
        The only case left is when $tr(u^5)=tr(v^4)$. In this case $[pr_{Ax(v)}(*), v^4.pr_{Ax(v)}(*)]$ must be contained in the intersection $Ax(u)\cap Ax(v)$ (see Figure \ref{Hyperbolic-Hyberbolic2}). In particular, there is an edge, the first edge between $v^2.pr_{Ax(v)}(*)$ and $v^3.pr_{Ax(v)}(*)$, that is fixed by $[u,v]$. Hence $[u,v]=1$, a contradiction.

    \begin{figure}[ht!]
    \centering
    \includegraphics[width=.8\textwidth]{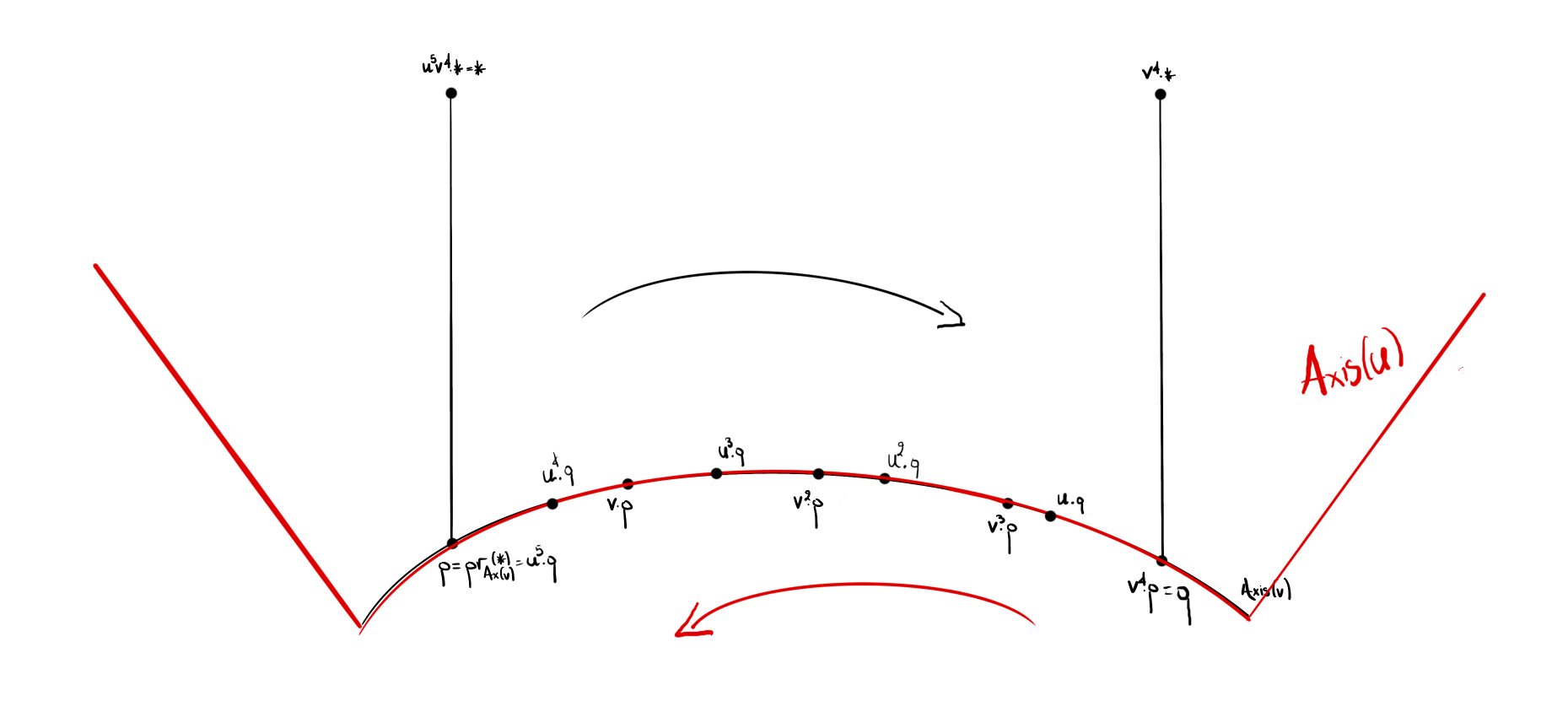}
    \caption{The hyperbolic case when $|Ax(u)\cap Ax(v)|>min\{tr(u^5),tr(v^4)\}$.}
    \label{Hyperbolic-Hyberbolic2}
    \end{figure}
        
        \item[(ii)] Suppose $|Ax(u)\cap Ax(v)|< min\{tr(u^5), tr(v^4)\}$. Then,  $tr(u^5v^4)=tr(u^5)+tr(v^4)-2|Ax(u)\cap Ax(v)|>0$. Consequently,  $u^5v^4$ is hyperbolic. 

        \item[(iii)] Suppose $|Ax(u)\cap Ax(v)|=min\{tr(u^5),tr(v^4)\}$. We first assume that $tr(v^4)\leq tr(u^5)$ and $tr(u^5v^4)=0$. Then the unique fixed point of $u^5v^4$ must be $*$ and it lies in the axis of $u$ (see Figure \ref{Hyperbolic-Hyberbolic3}). If $u.v^4.*$ is within $[v^4.r, r]=Ax(u)\cap Ax(v)$ (and, consequently, we must have that $u^4.v^4.*$ is within the intersection of the axes), then, as before, $[u,v]$ fixes an edge that lies within the intersection (e.g. an edge starting at $u^2.v^4.*$), hence $[u,v]$ is trivial, a contradiction. If not, then $u^3.v^4.u^3.v^4$ fixes the first edge, say $e$, of the segment $[u^4.v^4.*, r]$. Thus, $(u^3v^4)^2$ is trivial. On the other hand, $u^3v^4$ fixes $u^4v^4.*$ and since the stabilizer of $*$ is $\F_{\omega}$, we get that $u^3v^4$ is a conjugate of an element of $\F_\omega$ by $u^4v^4$. In particular, since $(u^3v^4)^2$ is trivial, we must have that $(u^3v^4)$ is trivial, again a contradiction.

    \begin{figure}[ht!]
    \centering
    \includegraphics[width=.8\textwidth]{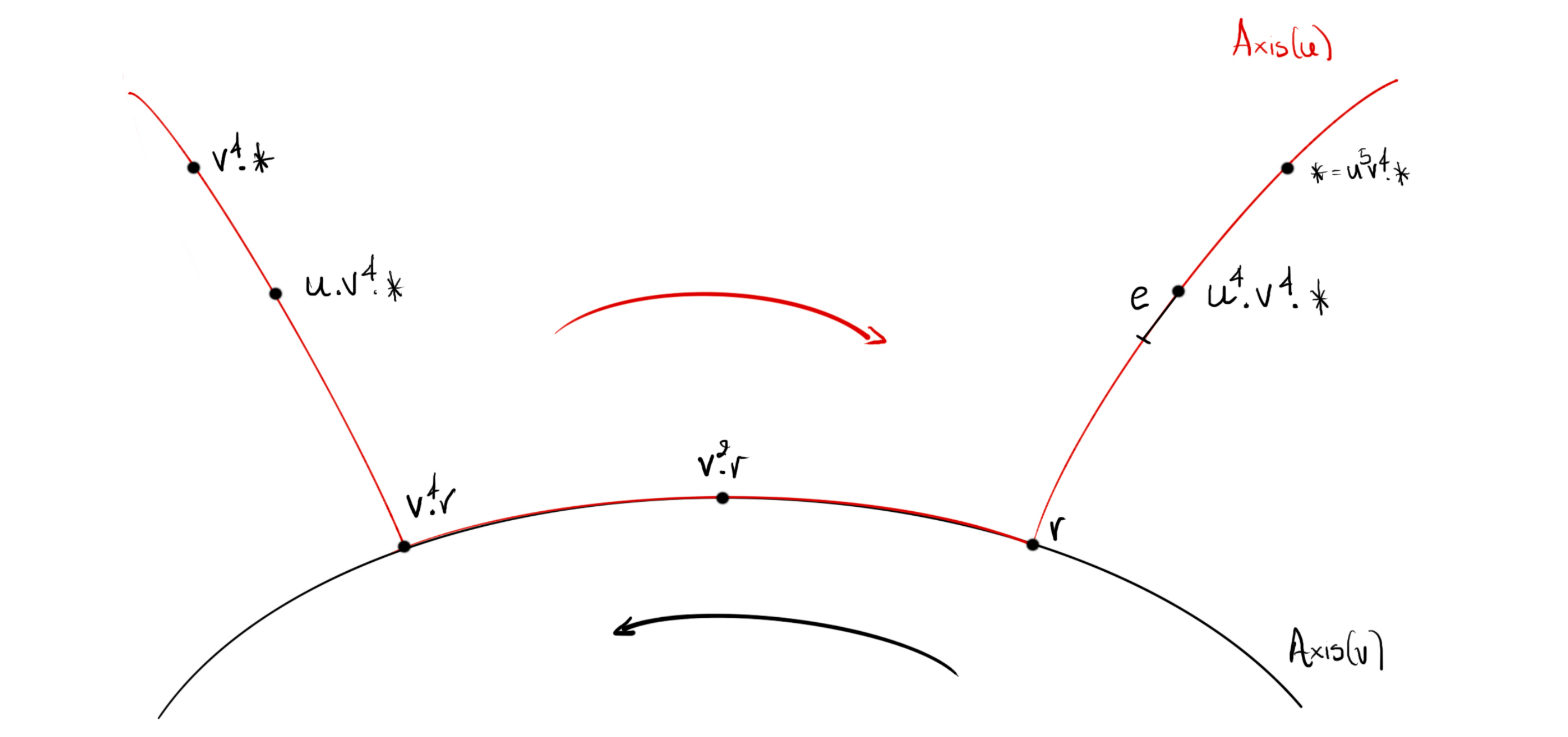}
    \caption{The hyperbolic case when $|Ax(u)\cap Ax(v)|=min\{tr(u^5),tr(v^4)\}$ and $tr(v^4)\leq tr(u^5)$.}
    \label{Hyperbolic-Hyberbolic3}
    \end{figure}

        Finally, we assume that $tr(u^5) < tr(v^4)$ and $tr(u^5v^4)=0$. This case is symmetric to the previous one. The unique fixed point of $u^5v^4$ must be $*$ and it lies in the axis of $v$. If $v.*$ belongs to $[pr_{Ax(u)}(v^4.*), u^5.pr_{Ax(u)}(v^4.*)]$, then there is an edge within this segment that is fixed by $[u,v]$, a contradiction. If not, then $v^{-2}.u^{-5}.v^{-2}.u^{-5}$ fixes the first edge of the segment $[v.*, r]$ (see Figure \ref{Hyperbolic-Hyberbolic4}). Thus, $(v^{-2}u^{-5})^2$ is trivial. On the other hand, $v^{-2}u^{-5}$ fixes $v.*$. The latter implies that $v^2u^5$ is a conjugate of an element of $\F_\omega$ by $v$. In particular, since $(v^{-2}u^{-5})^2$ is trivial, we must have that $v^{-2}u^{-5}$ is trivial, again a contradiction.

    \begin{figure}[ht!]
    \centering
    \includegraphics[width=.8\textwidth]{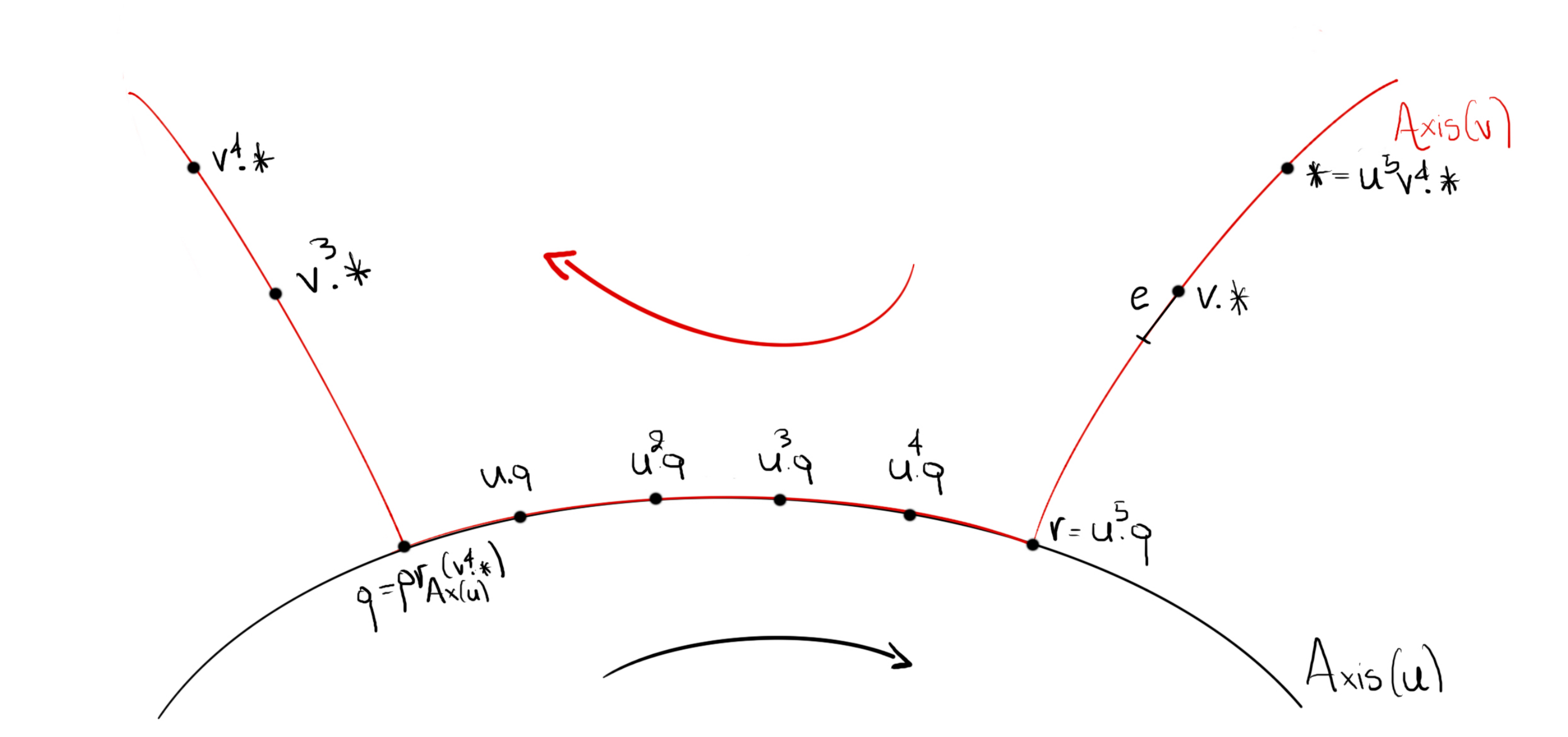}
    \caption{The hyperbolic case when $|Ax(u)\cap Ax(v)|=min\{tr(u^5),tr(v^4)\}$ and $tr(u^5)<tr(v^4)$.}
    \label{Hyperbolic-Hyberbolic4}
    \end{figure}

    \end{itemize}


  
\end{itemize}
\end{proof}
\newpage
\bibliography{biblio}

 \end{document}